\documentclass[12pt]{amsart}

\usepackage[utf8]{inputenc}
\usepackage[T1]{fontenc}

\usepackage[dvipsnames]{xcolor}
\usepackage{amsmath, amsfonts, amsthm, amssymb}
\usepackage[top=1in, bottom=1in, left=1in, right=1in]{geometry}
\usepackage[backref=page, colorlinks=true, allcolors=blue]{hyperref}
\usepackage[alphabetic,lite,backrefs]{amsrefs} 
\usepackage{tikz}
\usepackage{changepage}
\usepackage{mathrsfs}
\usepackage{tikz-cd}
\usepackage{mathtools} 
\usepackage{bm} 
 
\usepackage{enumitem} 

\usepackage[colorinlistoftodos,prependcaption,textsize=scriptsize]{todonotes}

\newcommand{\psuedo}[1]{\todo[disable]{#1}}
\makeatletter
\providecommand\@dotsep{5}
\renewcommand{\listoftodos}[1][\@todonotes@todolistname]{%
  \@starttoc{tdo}{#1}}
\makeatother

\theoremstyle{plain}
	\newtheorem{theorem}{Theorem}[section]
	\newtheorem{lemma}[theorem]{Lemma}
    
    \newtheorem{proposition}[theorem]{Proposition}
    
    \newtheorem{fact}[theorem]{Fact}
\theoremstyle{definition}
    \newtheorem{defn}[theorem]{Definition}
    \newtheorem*{defn*}{Definition}
    \newtheorem{example}[theorem]{Example}
    \newtheorem*{example*}{Example}
    
\theoremstyle{remark}
	\newtheorem{remark}[theorem]{Remark}
	\newtheorem*{remark*}{Remark}

\usepackage{mathtools}
\newcommand\SetSymbol[1][]{\nonscript\:#1\vert\allowbreak\nonscript\:\mathopen{}}
\providecommand\given{} 
\DeclarePairedDelimiterX\Set[1]\{\}{\renewcommand\given{\SetSymbol[\delimsize]}#1}

\newcommand{\ms}[1]{\mathscr{#1}}
\newcommand{\mc}[1]{\mathcal{#1}}
\newcommand{\coker}[0]{\text{coker }}

\newcommand{\res}[2][]{\left.{#2}\right|_{#1}}

\def\r{{\bm{r}}}
\def\m{{\bm{m}}}

\def\Z{{\mathbb{Z}}}
\def\R{{\textbf{R}}}
\def\N{{\mathbb{N}}}
\def\P{{\mathbb{P}}}
\def\k{{\Bbbk}}
\def\F{{\ms F}}

\def\L{{\ms{L}}}
\def\M{{\mc{M}}}
\def\A{{\mc{A}}}
\def\I{{\mc{I}}}
\def\E{{\mathscr{E}}}

\def\N{{\mathbb N}}

\def\O{{\mathcal O}}

\DeclareMathOperator{\Pic}{Pic}

\DeclareMathOperator{\rank}{rank}
\DeclareMathOperator{\reg}{reg}
\DeclareMathOperator{\Fitt}{Fitt}
\DeclareMathOperator{\Supp}{Supp}
\DeclareMathOperator{\Sym}{Sym}
\DeclareMathOperator{\Hom}{Hom}
\renewcommand{\Im}{\text{Im}\,}
\newcommand{\inj}{\hookrightarrow}

\title{Syzygies of Curves in Products of Projective Spaces}

\author{John Cobb}
\address{Department of Mathematics, University of Wisconsin, Madison, WI}
\email{jcobb2@math.wisc.edu}

\begin{document}

\maketitle
\vspace{-0.3in}
\begin{abstract}
    Motivated by toric geometry, we lift machinery for understanding the syzygies of curves in projective space to the setting of products of projective spaces. Using this machinery, we show an analogue of an influential result of Gruson, Peskine, and Lazarsfeld that gives a bound on the regularity of a possibly singular curve given its degree and the dimension of the ambient projective space. To do so, we show new results linking the shape of multigraded resolutions of a sheaf to its regularity region.
\end{abstract}

\setcounter{section}{-1}

\section*{Introduction}
\renewcommand*{\thetheorem}{\Alph{theorem}}

Understanding the syzygies among the equations defining a projective variety continues to be instrumental in understanding their properties. Classical theorems investigating when a variety is projectively normal or cut out by quadrics now fit into a more general syzygy program \cite{Green1984KoszulCA}, and even partial information about a variety's syzygies (e.g. their ranks or degrees) can give valuable geometric information. A celebrated result of Gruson, Peskine, and Lazarsfeld gives optimal bounds on how ``algebraically complicated" the syzygies of a degree $d$ curve $C\subseteq \P^r$ can be in terms of $r$ and $d$ \cite{gruson1983theorem}. This result settled classical questions of both Castelnuovo and Mumford and gave initial evidence towards the Eisenbud-Goto Conjecture, at least when the singularities are controlled. More importantly, the techniques and ideas developed in Gruson, Lazarsfeld, and Peskine's work were utilized either explicitly or implicitly to prove a number of other statements -- e.g. in a simple proof of Petri's theorem \cite{SimpleproofPetri} and consequently in the geometric syzygy conjecture in even genus \cite{Kemeny2020UniversalSB}, and in criteria for $N_p$ conditions on line bundles on smooth complex projective varieties \cite{Ein1993SyzygiesAK}. 

In recent years, several researchers have aimed to extend syzygy tools to varieties embedded in more general toric varieties \cites{Hoffman2002CastelnuovoMumfordRI, MaclaganSmith04, SidmanVantuyl06, Costa2006mBlocksCA, Hering2006SyzygiesMR, Huy07,HuyStrunk07, Hering10, Yang19, Bruce19, ChardinNemati20,berkesch2017virtual, Bruce20, BrownErman21, BrownSayrafi22}. One of the main innovations in \cite{gruson1983theorem} is recognizing that the positivity of the \textit{kernel bundle} coming from $C \inj \P^r$ controls a complex that nearly resolves the ideal sheaf $\I_C$ and thus governs the curve syzygies. For a recent example, this was critical to classify the singularities and study syzygies of the secant varieties of smooth curves \cite{Ein_2020}. The purpose of this paper is to lift this machinery to the case of products of projective spaces and as an example, use it to prove a bound on regularity for an integral curve in an arbitrary product of projective spaces (Proposition \ref{regularityCor}). Such a bound has been explored prior; Lozovanu gave a bound on the regularity on a smooth curve in a product of two projective spaces, as long as neither are $\P^1$ \cite{Lozovanu2008RegularityOS}.

We write $\P^\r \coloneqq \P^{r_1} \times \cdots \times \P^{r_n}$ for the product of projective spaces over the algebraically closed field $\k$, where $\r = (r_1, \dots, r_n) \in \N^n$. We take the magnitude of $\r$ to be $|\r|=r_1 + \cdots + r_n$. For any $\m = (m_1,\dots, m_n)\in \Pic(\P^\r) = \Z^n$ we denote the corresponding line bundle $\O_{\P^{r_1}}(m_1) \boxtimes \cdots \boxtimes \O_{\P^{r_n}}(m_n)$ by $\O_{\P^\r}(\m)$. Stated in the smooth case, our main result gives a resolution of the ideal sheaf of a curve embedded in a product of projective spaces:
\begin{theorem}\label{maintheorem}
    Let $C\subseteq \P^\r$ $(n>1)$ be a smooth nondegenerate curve of multidegree $\bm{d} = (d_1,\dots, d_n)\in \N^n$. Setting $a \coloneqq \max\Set*{d_i+d_j - r_i - r_j \given i\neq j} + 2$, then the ideal sheaf $\I_C$ admits a multigraded free resolution:
    \begin{equation*}\tag{$*$}
        \cdots \to \E_2 \to \E_1 \to \E_0 \to \I_C \to 0
    \end{equation*}
    such that each $\E_i$ is a direct sums of line bundles $\O_{\P^\r}(-\m)$ with $\m\in \N^n$ such that $|\m| = a + i$. In particular, the curve $C$ is cut out scheme-theoretically by multigraded polynomials of multidegree of magnitude at most $a$.
\end{theorem}
\begin{example*}
    Let $p:\P^1\inj \P^2\times \P^2$ be given by
    \begin{equation*}
        p([s:t]) = [t^2s-4s^3:t^3-4s^2t:t^2s-3s^3] \times [s^2t - t^3:s^3-st^2:t^3].
    \end{equation*}
    This is nondegenerate in each factor and defines a curve $C$ of degree $(3,3)$. Theorem \ref{maintheorem} gives us an explicit resolution
    \begin{equation*} \scriptsize
        \cdots 
        \to
        \begin{matrix}
            \O_{\P^2\times \P^2}(-5,-2)^{10}\\
            \oplus\\
            \O_{\P^2\times \P^2}(-4,-3)^{50}\\
            \oplus\\
            \O_{\P^2\times \P^2}(-3,-4)^{50}\\
            \oplus\\
            \O_{\P^2\times \P^2}(-2,-5)^{10}
        \end{matrix}
        \to
        \begin{matrix}
            \O_{\P^2\times \P^2}(-5,-1)^{5}\\
            \oplus\\
            \O_{\P^2\times \P^2}(-4,-2)^{50}\\
            \oplus\\
            \O_{\P^2\times \P^2}(-3,-3)^{100}\\
            \oplus\\
            \O_{\P^2\times \P^2}(-2,-4)^{50}\\
            \oplus\\
            \O_{\P^2\times \P^2}(-1,-5)^{5}
        \end{matrix}
        \to 
        \begin{matrix}
            \O_{\P^2\times \P^2}(-5,0)\\
            \oplus\\
            \O_{\P^2\times \P^2}(-4,-1)^{25}\\
            \oplus\\
            \O_{\P^2\times \P^2}(-3,-2)^{100}\\
            \oplus\\
            \O_{\P^2\times \P^2}(-2,-3)^{100}\\
            \oplus\\
            \O_{\P^2\times \P^2}(-1,-4)^{25}\\
            \oplus\\
            \O_{\P^2\times \P^2}(0,-5)
        \end{matrix}
        \to 
        \begin{matrix}
            \O_{\P^2\times \P^2}(-4,0)^{5}\\
            \oplus\\
            \O_{\P^2\times \P^2}(-3,-1)^{50}\\
            \oplus\\
            \O_{\P^2\times \P^2}(-2,-2)^{100}\\
            \oplus\\
            \O_{\P^2\times \P^2}(-1,-3)^{50}\\
            \oplus\\
            \O_{\P^2\times \P^2}(0,-4)^{5} \\
        \end{matrix}
        \to \I_C \to 0.
    \end{equation*}
    The first surjection immediately implies that $C$ can be cut out scheme-theoretically by equations of magnitude $\leq 4$.
\end{example*}
\psuedo{To apply Theorem \ref{maintheorem}, we wish to find a line bundle $\O_{\P^1}(d)$ on $\P^1$ witnessing the vanishings
    \begin{equation*}
        H^1(\P^1, p^*\O_{\P^2}(-1) \otimes \O_{\P^1}(d)) = H^1(\P^1, p^*\Omega_{\P^2}(1) \otimes p^*\Omega_{\P^2}(1) \otimes \O_{\P^1}(d)) = 0.
    \end{equation*}
    Leveraging our understanding of line bundles on $\P^1$, we can explicitly compute the pullbacks above. Since $p$ has degree 3, $p^*\O_{\P^2}(-1) = \O_{\P^1}(-3)$. To compute $p^*\Omega_{\P^2}(1)$, recall that it fits into the pullback of the Euler sequence on $\P^2$:
    \begin{equation*}
        0 \to p^*\Omega_{\P^2}(1) \to \O_{\P^1}^{\oplus 3} \to \O_{\P^1}(3) \to 0
    \end{equation*}
    This establishes $c_1(\Omega_{\P^2}(1)) = -3$. On the other hand, we have the exact sequence 
    \[ 0 \to \O_{\P^1}(-2) \to p^* \Omega_{\P^2}(1) \to \O_{\P^1}(-1) \to 0 \] 
    where $\O_{\P^1}(-2) = \I_C/\I_C^2$ is the conormal sheaf associated to the map $p$. This short exact sequence necessarily splits, given that the extension class lives in 
    \[ H^1(\P^1, \Hom(\O_{\P^1}(-1), \O_{\P^1}(-2))) = H^1(\P^1, \O_{\P^1}(-1)) = 0. \]
    Thus $p^*\Omega_{\P^2}(1) = \O_{\P^1}(-2) \oplus \O_{\P^1}(-1)$. Plugging these calculations back into our desired vanishing statements, one can check that choosing $d=3$ suffices. } 

Theorem \ref{maintheorem} generalizes the main construction appearing in the proof of \cite{gruson1983theorem}*{Proposition 1.2} and gets a bound as in \cite{gruson1983theorem}*{Lemma 1.7}. As will be implied by Proposition \ref{maintheoremsingular}, the final observation in Theorem \ref{maintheorem} holds even if $C$ is singular in which case $(*)$ is a complex that locally resolves $\I_C$ off the singular locus of $C$. The proof strategy shares a similar outline to \cite{gruson1983theorem}, although new roadblocks are introduced in the multigraded setting. Supposing that $C$ is smooth, the essential idea is to express $\O_C$ as the cokernel of a linear map. Having done this, the associated Eagon-Northcott complex gives the resolution $(*)$. The main hurdle is constructing this linear map -- it is realized as a derived pushforward of a resolution of the graph of $C \inj \P^\r$ obtained by restricting a product of Be\u{i}linson's resolutions. As in \cite{gruson1983theorem}, the critical observation is that the exactness of the terms in the derived pushforward depends only on the cohomological vanishing of mixed wedge products of a multigraded kernel bundle. If $C$ is not smooth, we may use essentially the same argument on its smooth normalization. The new analysis required to bound $a$ is why the case of projective space ($n=1$) is specifically excluded from Theorem \ref{maintheorem}. As we will note in Remark \ref{Why statement doesnt reduce to n=1}, this is because computing the derived pushforward mentioned before involves computing the first two terms of a Be\u{i}linson spectral sequence, and the second of those terms behaves differently in the cases $n = 1$ and $n > 1$. 

It is natural to ask whether something similar to Theorem \ref{maintheorem} could be investigated for curves in more general toric varieties. There are several obstacles to such an extension. Most notably, there is not a strong multigraded analogue of Be\u{i}linson's resolution in the literature. Recent work of Hanlon-Hicks-Lazarev provides a generalization of Be\u{i}linson's resolution to all toric stacks, but they are not as explicit or simple \cites{hanlon2023resolutions,brown2023short}. Moreover, the resolution in $(*)$ results from taking the associated Eagon-Northcott of a linear matrix, and in the general toric case there are several potential notions of linearity to consider \cite{brown2023linear}.

In the case of standard projective space, the existence of a linear resolution as in Theorem \ref{maintheorem} is equivalent to a certain bound on regularity. In the case of a product of $n$ projective spaces, the standard notion of (multigraded) regularity introduced by Maclagan and Smith is a region inside $\Pic(\P^\r) = \Z^n$ given as a union of cones of type $\m + \N^n$ (see Definition \ref{multigradedregdef}) \cite{MaclaganSmith04}. As has been shown in many previous works, the connection between the shape of free resolutions and regularity is much weaker in the multigraded case \cites{Eisenbud2014TateRF, berkesch2017virtual,SayrafiBruceMultigradedRegularity, BotbolChardin17,ChardinHolanda}. 
In \S \ref{Multigraded Regularity section}, new results linking the shape of resolutions to multigraded regularity are proven. In particular, by applying Theorem \ref{lineartoreg} to the complex $(*)$ we obtain a bound on multigraded regularity akin to the main result of \cite{gruson1983theorem}*{Theorem 1.1}:
\begin{proposition}\label{regularityCor}
    Let $C\subseteq \P^\r$ $(n>1)$ is an integral nondegenerate curve of degree $\bm{d}$. Setting the constant $a$ as in Theorem \ref{maintheorem}, then as long as $\P^\r\neq \P^1\times \P^1$ we have 
    \begin{equation*}
        (a_1,\dots, a_n) + \N^n\subseteq \reg(\I_C)
    \end{equation*} 
    where $a_k = \min\{r_ka-d_k,a\}$.
\end{proposition}

For example, Proposition \ref{regularityCor} says that a curve $C$ in $\P^{\bm{1}}$ of degree $(d,\dots, d)$ is $(d,\dots, d)$-regular for $n>2$, and that the curve in the previous example is $(4,4)$-regular.\footnote{One can compute using \textit{Macaulay2} that the actual regularity region is $(2,2) + \N^2$.} Due to standard properties of multigraded regularity, there are various corollaries one can say about $C$. For instance, as long as $\bm{d}\in \reg(\I_C)$ then the complete linear series associated to $\O_C(\bm{d})$ gives a projectively normal embedding of $C$, provided that each coordinate of $\bm{d}$ is positive. Perhaps surprising is that the dominating term $a$ in Proposition \ref{regularityCor} only depends on the sum of data coming from the behavior of the curve in only two (rather than all) of the products in $\P^\r$ as suggested in \cite{lozovanu2012vanishing}.

At least in the smooth and $n=2$ case, analogues like Proposition \ref{regularityCor} have been studied by Lozovanu \cite{Lozovanu2008RegularityOS} using generic projection techniques developed by Gruson and Peskine \cites{GP1,GP2}. These techniques avoid a generalization such as Theorem \ref{maintheorem} and use a stronger notion of nondegenerate\footnote{The main bounds in \cite{Lozovanu2008RegularityOS} are for curves embedded in $\P^{d_1} \times \P^{d_2}$ where $d_1,d_2 > 1$ with birational projections onto each factor.}, but gives sharper bounds -- further comparisons are made in Examples \ref{standardexample} and \ref{LozovanuExample}. 

This paper is organized as follows: In \S \ref{Background and Notation} we provide notation and necessary background. In \S \ref{Multigraded Regularity section}, we prove new results linking the shape of resolutions to multigraded regularity and introduce some of the tools we will need for Theorem \ref{maintheorem}. In \S \ref{Main Result and Proof}, Theorem \ref{maintheorem} is proven, and then in \S \ref{Regularity of Curves} it is used to prove Proposition \ref{regularityCor}. Lastly, in \S \ref{Examples} we explore more examples of the utility of our result.

\subsection*{Acknowledgements} Thanks to Daniel Erman for his valuable insight during this project. Thanks also to Michael Kemeny who helped me understand some of the arguments in \cite{gruson1983theorem}, to Rob Lazarsfeld for alerting me to a different treatment of the $n=2$ case by his student in \cite{Lozovanu2008RegularityOS}, and to an anonymous reviewer for their thorough reading and detailed suggestions that greatly improved the paper. The computer algebra system \textit{Macaulay2} \cite{M2} was used extensively, in particular, the \texttt{VirtualResolutions} package \cite{Almousa_2020}.

\section{Background and Notation}\label{Background and Notation}
\renewcommand*{\thetheorem}{\thesection.\arabic{theorem}}
\begin{enumerate}[label = (\thesection.\theenumi)]
    \item We work throughout over an algebraically closed field $\k$ of arbitrary characteristic.
    \item We let $\bm{e}_1, \dots, \bm{e}_n$ be the standard basis of $\Pic(\P^\r) \cong \Z^n$ and order the elements component-wise. To extend the notation in the introduction, given a sheaf $\F$ on $\P^\r$ defined by
    \[ \F = \F_1 \boxtimes \cdots \boxtimes \F_n = \pi_1^* \F_1 \otimes \cdots \otimes \pi_n^* \F_n \]
    where $\pi_i: \P^\r \to \P^{r_i}$ is the $i$th projection map, we define the twist of $\F$ by $\m \in \Pic(\P^r)$ as $\F(\m) \coloneqq \F_1(m_1) \boxtimes \cdots \boxtimes \F(m_n)$ and the the exterior power of $\F$ by $\m$ as $\wedge^\m \F \coloneqq \wedge^{m_1} \F_1 \boxtimes \cdots \boxtimes \wedge^{m_n} \F_n$.
    \item \label{def:nondegeneratecurve} A \textit{curve} is a reduced and irreducible, but possibly singular, projective variety of dimension one. We will call the curve \textit{nondegenerate} if the inclusion $C\subseteq \P^\r$ is nondegenerate in each factor. Note that this is equivalent to asking that no multigraded form of magnitude 1 appears in the defining ideal.
    \item \label{kernelbundle} An embedding $p\colon C \inj \P^\r$ is determined componentwise via the linear systems of $n$ line bundles $\L_i$ on $C$. That is, $p_i : C \to \P^{r_i}$ is associated to an linear subspace $V_i \subseteq H^0(C,\L_i)$ of dimension $r_i+1$ (since $p$ is nondegenerate). This gives rise to the exact sequence 
    \begin{equation*}
        0 \to \M_{V_i} \to V_i \otimes \O_C \stackrel{\text{ev}}{\to} \L_i \to 0.
    \end{equation*}
    $\M_{V_i}$ is the \textit{kernel bundle} associated to $V_i$ of rank $r_i$. More concretely, we can think of the short exact sequence above as being the pullback via $p_i$ of the Euler sequence 
    \begin{equation*}
        0 \to \Omega_{\P^{r_i}}(1) \to (\O_{\P^{\bm{r}}})^{\oplus r_i+1} \to \O_{\P^{r_i}}(1) \to 0.
    \end{equation*}
    One can obtain a natural multigraded analogue to the kernel bundle by taking the direct sum of the complexes above for each factor and pulling back by $p$. We will be primarily concerned with the cohomology of the subbundle $\M \coloneqq p^* \Omega_{\P^\r}(\bm{1})$, or more accurately, of 
    \begin{align*}
         \wedge^\m \M &= \wedge^\m (p^* \Omega_{\P^\r}(\bm{1})) \\
         &= \wedge^\m p^* \Omega_{\P^\r}(\m)
    \end{align*}
    which we denote by $p^*\Omega^\m_{\P^\r}(\m)$.
    \item \label{EagonNorthcott} Given a homomorphism $\varphi\colon \E \to \F$ between two vector bundles of ranks $e\geq f$ on a smooth variety, we can define the Eagon-Northcott complex
    \begin{multline*}
        0 \to \bigwedge^e \E \otimes (\Sym_{e-f}\F)^* \to \bigwedge^{e-1} \E \otimes (\Sym_{e-f+1}\F)^* \to \cdots\\
        \cdots \to \bigwedge^{f+1} \E \otimes (\Sym_1\F)^* \to \bigwedge^f \E\otimes (\Sym_0\F)^* \stackrel{\wedge^f \varphi}{\longrightarrow} \bigwedge^f \F \to 0  
    \end{multline*}
    The image of $\wedge^f \varphi$ is $\Fitt(\coker \varphi)$, the Fitting ideal sheaf of $\coker \varphi$ corresponding to the ideal of maximal minors of $\varphi$. One basic fact we need to use about this complex is as follows:
    
    \begin{fact}[c.f. Appendix B.2  \cite{lazarsfeld2017positivity}]\label{ENfact}
        The Eagon-Northcott complex is always exact away from $\Supp(\coker \varphi)$. If $C$ is smooth the Eagon-Northcott is exact.
    \end{fact}
\end{enumerate}

\section{Multigraded Regularity}\label{Multigraded Regularity section}
A coherent sheaf $\F$ on $\P^r$, is $m$-regular if 
\[ H^i(\P^r, \F \otimes \O_{\P^r}(m-i)) = 0, \hspace{2em} \forall i > 0. \] The Castelnuovo-Mumford regularity $\reg(\F)$ of $\F$ is then defined as the least integer $m$ for which $\F$ is $m$-regular. Castelnuovo-Mumford regularity is a fundamental invariant within algebraic geometry that can be thought to approximate the smallest twist for which the sheaf is generated by its global sections. It has been used to bound the degree of syzygies \cites{gruson1983theorem,Ein1993SyzygiesAK, Kwak1998CastelnuovoRF}, construct Hilbert schemes \cites{Gotzmann1978EineBF,FGAExplained}, and understand the algorithmic complexity of computing Gr\"{o}bner bases for syzygies \cite{Bayer1988OnTC}.

Due mostly to work in toric geometry, interest in understanding a multigraded analogue to regularity is burgeoning, e.g. \cites{H2002TheRO, Maclagan2003UniformBO, Hering2006SyzygiesMR,Lozovanu2008RegularityOS, Eisenbud2014TateRF, BotbolChardin17,ChardinNemati20, BerkeschKleinLoperYang21,  SayrafiBruceMultigradedRegularity,Bruce2022BoundsOM}. Maclagan and Smith \cite{MaclaganSmith04} introduced a standard notion of multigraded regularity:
\begin{defn}\label{multigradedregdef}
    A coherent sheaf $\F$ on $\P^\r=\P^{r_1} \times \cdots \times \P^{r_n}$ is $\m$-\textit{regular} if 
    \begin{equation*}
        H^{|\bm{i}|}(\P^\r, \F \otimes \O_{\P^\r}(\m-\bm{i})) =0,
    \end{equation*}
    for all $\bm{i}\in \N^n$ such that $|\bm{i}| > 0$. The (multigraded) \textit{regularity} $\reg(\F)$ of $\F$ is the region 
    \[ \reg(\F) = \Set{ \m \in \Z^n \given \F \text{ is } \m\text{-regular}}. \]
\end{defn}
For example, the sheaf $\O_{\P^\r}(-a,-b)$ is $(a,b)$-regular and the sheaf $\O_{\P^\r}(-1,2) \oplus \O_{\P^r}(1,-3)$ is $(1,3)$-regular. In the standard graded case, Definition \ref{multigradedregdef} reduces to the Castelnuovo-Mumford regularity. Further, many of the standard facts and geometric interpretations about Castelnuovo-Mumford regularity generalize to multigraded regularity. For instance, letting $X$ be a toric variety then:
\begin{enumerate}
    \item[(i)] If $\F$ is $\m$-regular then $\m+\N^n\subseteq \reg \F$.
    \item[(ii)] If $\I_X$ is the ideal sheaf of $X$ and $\m\in \reg(\I_X)$ then $X$ is cut out scheme theoretically by equations of degree $\m$.
    \item[(iii)] If $\O_X(\m)$ is an ample line bundle and $\m\in \reg(\O_X)$ then the complete linear series associated to $\O_X(\m)$ gives a projectively normal embedding of $X$.
\end{enumerate}
Unfortunately, many of the standard tools we have for studying Castelnuovo-Mumford regularity fail when passing to multigraded regularity. For instance, it is possible to find two resolutions with the same Betti numbers but different multigraded regularities (see Example 5.1 in \cite{SayrafiBruceMultigradedRegularity}). 
There have nevertheless been several efforts to prove analogues of the equivalence between multigraded regularity and Betti numbers. Maclagan and Smith proved the first such result in \cite{MaclaganSmith04}*{Theorem 1.7(1)}, which was sharpened slightly by the work of Botbol and Chardin \cite{BotbolChardin17}. Berkesch, Erman, and Smith then recovered a variant of the equivalence if one worked with \textit{virtual resolutions} \cite{berkesch2017virtual}*{Theorem 2.9}. This variant was made effective in \cite{SayrafiBruceMultigradedRegularity}*{Theorem A}. Virtual resolutions are actively studied: see \cites{Yang19,Loper2019WhatMA,GaoLiLoperMattoo21, BerkeschKleinLoperYang21, BoomsPeot2022}, among many others.

In our case, we need results estimating regularity in terms of the shape of the resolution but in cases where the module is generated in multiple distinct multidegrees. For these purposes, Maclagan and Smith's result is the most flexible. The following theorem accomplishes a generalization of this by allowing complexes that are exact off of a set of dimension $\leq 1$, combining the virtual resolution from \cite{MaclaganSmith04}*{Theorem 1.7(1)} with the idea that "near exactness" is enough for bounding regularity from \cite{gruson1983theorem}.

\begin{proposition}[Compare with Theorem 7.2 in \cite{MaclaganSmith04}]\label{MSGentheorem}
    Suppose the complex 
    \begin{equation*}
        \cdots \to \E_3 \stackrel{d_3}{\to} \E_2 \stackrel{d_2}{\to} \E_1 \stackrel{d_1}{\to} \E_0 \stackrel{d_0}{\to} \F \to 0 
    \end{equation*}
    of coherent sheaves on $\P^\r$ is exact away from a set of dimension $\leq 1$. Then we have 
    \begin{equation*}
        \bigcup_{\phi\colon [|\r|+1] \to [n]} \left(  \bigcap_{0\leq i \leq |\r|+1} \left(-\bm{e}_{\phi(1)} - \cdots -\bm{e}_{\phi(i)} + \reg(\E_i) \right) \right) \subseteq \reg(\F)
    \end{equation*}
    where the union is over all functions $\phi\colon[|\r|+1] \to [n]$ and $[k] \coloneqq \{1, \dots, k\}$.
\end{proposition}
\begin{example}
    Suppose that $\P^\r = \P^1\times \P^1$, so that $|\r| + 1 = 3$ and $n=2$. There are $8$ functions $\phi: [3] \to [2]$. For the function $\{1,2,3\} \stackrel{\phi}{\mapsto} \{1,1,1\}$, the corresponding term in the union is
    \[ \reg(\E_0) \cap ( (-1,0) + \reg(\E_1)) \cap ((-2,0) + \reg(\E_2)) \cap ((-3,0) + \reg(\E_3)). \]
\end{example}
\begin{proof}
    Fix a function $\phi \colon[|\r|+1] \to [n]$. We claim that for all $k\geq 0$,
    \begin{equation*}
        \bigcap_{k\leq i\leq |\r|+1} \left(-\bm{e}_{\phi(k+1)} - \cdots - \bm{e}_{\phi(i)} + \reg(\E_i)\right) \subseteq \reg(\Im d_k).
    \end{equation*}
    Since $\Im d_0 = \F$, this will prove the theorem. When $k>|\r|+1$, the intersection is empty so the claim holds trivially -- we will establish the claim by using descending induction on $k$. Suppose that $k\leq |\r|+1$ is an index so that the claim holds for $k+1$ and consider the exact sequence 
    \begin{equation*}
        0 \to \mathcal{H}_k(\E_\bullet) \to \E_k/\Im d_{k+1} \to \Im d_k \to 0,
    \end{equation*}
    where $\mathcal{H}_k(\E_\bullet)$ is the $k^{\text{th}}$ homology sheaf of $\E_\bullet$. Since $\E_\bullet$ is exact away from a set of dimension $\leq 1$, the homology sheaves are supported on a set of dimension $\leq 1$, therefore $H^i(\P^{\bm{r}},\mathcal{H}_k(\E_\bullet))$ vanishes for $i>1$. After twisting by $\m-\bm{i}$, the associated long exact sequence in sheaf cohomology implies that $H^{|\bm{i}|}(\P^\r, \E_k/\Im d_{k+1}(\m-\bm{i}))=H^{|\bm{i}|}(\P^\r, \Im d_k(\m-\bm{i}))$ for $|\bm{i}|>1$ and that $H^1(\P^\r, \E_k/\Im d_{k+1}(\m-\bm{e}_j))$ surjects onto $H^1(\P^r, \Im d_k(\m-\bm{e}_j))$. This directly implies 
    \begin{equation*}
        \reg(\E_k/\Im d_{k+1}) \subseteq \reg(\Im d_k).
    \end{equation*}
    By Lemma \ref{multireginSES} (proven next), the short exact sequence $0 \to \Im d_{k+1} \to \E_k \to \E_k/\Im d_{k+1} \to 0$ gives
    \begin{equation*}
        \left(\bigcup_{1\leq j \leq |\r|+1} \left(-\bm{e}_j + \reg(\Im d_{k+1})\right)\right) \cap \reg(\E_k) \subseteq \reg(\E_k/\Im d_{k+1})
    \end{equation*}
    However, by using our induction hypothesis we find that the left side contains
    \begin{equation*}
        \bigcup_{1\leq j \leq |\r|+1} -\bm{e}_j + \left(\bigcap_{k\leq i\leq |\r|+1} \left(-\bm{e}_{\phi(k+2)} - \cdots - \bm{e}_{\phi(i)} + \reg(\E_i) \right)\right),
    \end{equation*}
    which further contains the $\phi(k+1)^{\text{th}}$ term of the union:
    \begin{equation*}
        -\bm{e}_{\phi(k+1)} + \bigcap_{k\leq i\leq |\r|+1} \left(-\bm{e}_{\phi(k+2)} - \cdots - \bm{e}_{\phi(i)} + \reg(\E_i)\right).
    \end{equation*}
    Due to the fact that our regularity regions are cones, this contains \begin{equation*}
         \bigcap_{k\leq i\leq |\r|+1} \left(-\bm{e}_{\phi(k+1)} -\bm{e}_{\phi(k+2)} - \cdots - \bm{e}_{\phi(i)} + \reg(\E_i)\right).
    \end{equation*} 
    Stringing together the last five inclusions gives the claim.
\end{proof}
\begin{lemma}[Compare with Lemma 7.1(2) in \cite{MaclaganSmith04}]\label{multireginSES}
    Suppose that $0\to \F' \to \F \to \F'' \to 0$ is a short exact sequence of coherent sheaves on $\P^r$. Then 
    \begin{equation*}
        \left(\bigcup_{1\leq j \leq |\r|+1} \left(-\bm{e}_j + \reg(\F')\right)\right) \cap \reg(\F) \subseteq \reg(\F'')
    \end{equation*}
\end{lemma}
\begin{proof}
    Suppose that $\F$ is $\m$-regular and that $\F'$ is $(\m+\bm{e}_j)$-regular for some $\bm{e}_j$. After twisting by $\m-\bm{i}$ for some $1\leq |\bm{i}| \leq |\bm{r}|$, the associated long exact sequence contains 
    \begin{equation*}
        \cdots \to H^{|i|}(\P^\r, \F(\m - \bm{i})) \to H^{|i|}(\P^\r, \F''(\m - \bm{i})) \to H^{|i|+1}(\P^\r, \F'(\m - \bm{i})) \to \cdots .
    \end{equation*}
    Since $\F$ is $\m$-regular, the first term is zero. Since $\F'$ is $(\m+\bm{e}_j)$-regular, we also find that the third term vanishes:
    \begin{equation*}
        H^{|\bm{i}|+1}(\P^\r, \F'(\m-\bm{i})) = H^{|\bm{i}|+1}(\P^\r, \F'(\m+\bm{e}_j-(\bm{i}+\bm{e}_j))) = 0.
    \end{equation*}
    Thus the middle term vanishes and $\F''$ is $\m$-regular as desired.
\end{proof}

For a sharper bound assuming that $\E$ is a direct sum of line bundles, see  \cite{Bruce2022BoundsOM}*{Lemma 4.6}. One possible problem with Proposition \ref{MSGentheorem} is that the region can be hard to compute -- the number of intersections and unions grows as $\text{O}(|\r|n^{|r|+1})$. However, if the ``growth'' of the regularity regions is controlled we can simplify dramatically.

\begin{defn}
    Let 
    \begin{equation*}
        \cdots \to \E_3 \to \E_2 \to \E_1 \to \E_0 \to \F \to 0 
    \end{equation*}
    be a complex on $\P^\r$ with $\E_\bullet$ a direct sum of line bundles $\O_{\P^\r}(-\m)$. We say that $\E_\bullet$ has \textit{linear twist growth} if for all twists $\O_{\P^\r}(-\m_k)$ appearing among $\E_k$ there exists a twist $\O_{\P^\r}(-\m_0)$ of $\E_0$ such that $\m_k - \m_0 \in \N^n$ and  $0 \leq |\m_k - \m_0| \leq k.$
\end{defn}

\begin{figure}
    \newcommand{\makegrid}{
    \foreach \x in {-2,...,6}
    \foreach \y in {-2,...,6}
        { \fill[gray,fill=gray] (\x,\y) circle (1.5pt); }
    \draw[-,  semithick] (-2,0)--(6,0);
    \draw[-,  semithick] (0,-2)--(0,6);
    }
    \begin{tikzpicture}[scale=.5]
    \makegrid
  
    \fill[Black] (0,2) circle (6pt);
    \fill[Black] (3,1) circle (6pt);
    \filldraw[color=OrangeRed, fill=White] (0,5) circle (6pt);
    \filldraw[color=OrangeRed, fill=White] (1,4) circle (6pt);
    \filldraw[color=OrangeRed, fill=White] (2,3) circle (6pt);
    \filldraw[color=OrangeRed, fill=White] (3,2) circle (6pt);
    \filldraw[color=OrangeRed, fill=White] (0,4) circle (6pt);
    \filldraw[color=OrangeRed, fill=White] (0,3) circle (6pt);
    \filldraw[color=OrangeRed, fill=White] (1,3) circle (6pt);
    \filldraw[color=OrangeRed, fill=White] (1,2) circle (6pt);
    \filldraw[color=OrangeRed, fill=White] (2,2) circle (6pt);
    \filldraw[color=OrangeRed, fill=White] (3,4) circle (6pt);
    \filldraw[color=OrangeRed, fill=White] (3,3) circle (6pt);
    \filldraw[color=OrangeRed, fill=White] (4,1) circle (6pt);
    \filldraw[color=OrangeRed, fill=White] (4,2) circle (6pt);
    \filldraw[color=OrangeRed, fill=White] (4,3) circle (6pt);
    \filldraw[color=OrangeRed, fill=White] (5,1) circle (6pt);
    \filldraw[color=OrangeRed, fill=White] (5,2) circle (6pt);
    \filldraw[color=OrangeRed, fill=White] (6,1) circle (6pt);
    \end{tikzpicture}\quad
    \caption{Suppose that the shifts in $\E_0$ are $(0,2)$ and $(3,1)$, e.g. $\E_0 = \O_{\P^\r}(0,2) \oplus \O_{\P^\r}(3,1)$. The possible shifts for $\E_3$ are displayed in red unfilled circles.}
\end{figure}
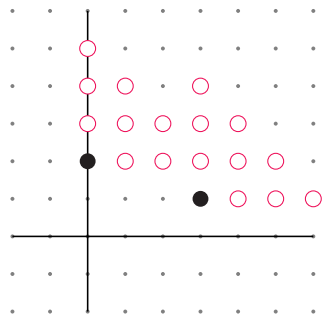

Equivalently, for any twist $\O_{\P^\r}(-\m_k)$ of $\E_k$ there exists $
\O_{\P^\r}(-\m_0)$ such that $\m_k = \m_0 + \bm{k}$ where $|\bm{k}|\leq k$. This definition generalizes the algebraic criterion for regularity, see \cite{lazarsfeld2017positivity}*{Theorem 1.8.26}. For example, Eagon-Northcott complexes have linear twist growth, and in particular, the resolution appearing in Theorem \ref{maintheorem} has linear twist growth. The following result, Theorem \ref{lineartoreg}, says that if a complex as in Proposition \ref{MSGentheorem} has linear twist growth then one can bound the regularity using only the first term. Applying this idea to Theorem \ref{maintheorem} is essentially the content of the proof of Proposition \ref{regularityCor}.

\begin{theorem}\label{lineartoreg}
    Let $\F$ be a coherent sheaf on $\P^\r$. Suppose the complex 
    \begin{equation*}
        \cdots \to \E_3 \to \E_2 \to \E_1 \to \E_0 \to \F \to 0 
    \end{equation*}
    is exact off of a set of dimension $\leq 1$ and $\E_i = \bigoplus_j \O_{\P^\r}(-\m_{i,j})$. If $\E_\bullet$ has linear twist growth then $\reg \E_0 \subseteq \reg \F$.
\end{theorem}
\begin{proof}
    Suppose $\E_0$ is $\m$-regular, we wish to show that $\F$ is too. By definition this means 
    \begin{align}\label{lineartoregcondition}
        H^{|\bm{i}|}(\P^\r, \E_0(\m-\bm{i})) &= H^{|\bm{i}|}(\P^\r, \bigoplus_j \O_{\P^\r}(-\m_{0,j}+\m-\bm{i}))\\
        &= \bigoplus_j H^{|\bm{i}|}(\P^\r, \O_{\P^\r}(-\m_{0,j}+\m-\bm{i}))\notag \\
        &= 0 \notag
    \end{align}
    for all $\bm{i} \in \N^n$ such that $|\bm{i}|>0$. Fix some $\bm{i}$ such that $1\leq |\bm{i}| \leq |\r|$ and twist our complex by $\m-\bm{i}$:
    \begin{equation}\label{lineartoregcomplex}
        \cdots \to \E_3(\m-\bm{i}) \stackrel{d_3}{\to} \E_2(\m-\bm{i}) \stackrel{d_2}{\to} \E_1(\m-\bm{i}) \stackrel{d_1}{\to} \E_0(\m-\bm{i}) \stackrel{d_0}{\to} \F(\m-\bm{i}) \to 0 
    \end{equation}
    Since $\E_\bullet$ has linear twist growth, we can inductively replace twists appearing in $\E_k(\m-\bm{i})$ as shown:
    \begin{align*}
        H^{|\bm{i}|+k}(\P^\r,\E_k(\m-\bm{i})) &= \bigoplus_j H^{|\bm{i}|+k}\left(\P^\r,\O_{\P^\r}(-\m_{k,j}+\m-\bm{i})\right)\\
        &= \bigoplus_j H^{|\bm{i}|+k}(\P^\r,\O_{\P^\r}(-\m_{0,j} + \m - (\bm{i} + \bm{k}_j))
    \end{align*}
    where $|\bm{k}_j|\leq k$. However, (\ref{lineartoregcondition}) now implies that $H^{|\bm{i}|+k}(\P^\r,\E_k(\m-\bm{i}))$ vanishes since $|\bm{i}+\bm{k}_j| \leq |\bm{i}|+k$. Therefore we are in the situation 
    \begin{equation*}
        H^{|\bm{i}|}(\P^\r,\E_0(\m-\bm{i})) = H^{|\bm{i}|+1}(\P^\r, \E_1(\m-\bm{i})) = \cdots = H^{|\bm{i}|+|\r|}(\P^\r, \E_{|\r|}(\m-\bm{i})) = 0
    \end{equation*}
    Furthermore, since complex (\ref{lineartoregcomplex}) is exact away from a set of dimension $\leq 1$, the homology sheaves are supported on a set of dimension $\leq 1$, therefore $H^i(\P^{\bm{r}},\mathcal{H}_k(\E_\bullet))$ vanishes for $i>1$. Therefore Proposition B.1.2 in \cite{lazarsfeld2017positivity} applies to give us the vanishing 
    \begin{equation*}
        H^{|\bm{i}|}(\P^\r, \F(\m-\bm{i})) =0.
    \end{equation*}
    for all $\bm{i} \in \N^n$ such that $|\bm{i}|>0$, as desired.
\end{proof}

When $r=1$, Theorem \ref{lineartoreg} specializes to the existence of linear resolutions on projective spaces; see Proposition 1.8.8 in \cite{lazarsfeld2017positivity}. \cites{MaclaganSmith04,berkesch2017virtual, SayrafiBruceMultigradedRegularity} all have results connecting multigraded regularity to the existence of certain invariants of linear resolutions, but those results are most applicable to the case where the first term is generated in a single multidegree. Some of the ways the notions from these papers compare are further explained in Remark \ref{Comparison of linearity}.

Notice that since $\E_0$ is a direct sum of line bundles, Theorem \ref{lineartoreg} will only ever bound $\reg(\F)$ by a region of the form $\m+\N^n$. Despite this apparent loss, it is worth pointing out that in the case of linear twist growth, Proposition \ref{MSGentheorem} reduces exactly to Theorem \ref{lineartoreg}. To end off this section, the following example illustrates how you might use Theorem \ref{lineartoreg} to estimate the regularity region of some subvariety of $\P^\r$.

\psuedo{
    Certainly, it is the case that the region is contained in $\reg(\E_0)$. All that is left to show is that $\reg(\E_0)$ is contained in the union of intersections of regions $\reg_{\E_i}$ for $i\geq 1$ in the proposition. We can do this because the linear twist growth condition implies (but is not the same as) that $\reg(\E_i) \subseteq \bigcap -\bm{e}_i +\reg{\E_{i+1}}$. 
}

\begin{example}\label{twopointsinP1xP1}
    Let $X$ be two generic points in $\P^1\times \P^1$. Then one resolution of $\I_X$ is given by
    \begin{equation*}
        0 \to
        \begin{matrix}
            \O_{\P^1 \times \P^1}(-2,-1)^2
        \end{matrix} 
        \to 
        \begin{matrix}
            \O_{\P^1 \times \P^1}(-2,0)\\ \oplus \\ 
            \O_{\P^1 \times \P^1}(-1,-1)^2
        \end{matrix} 
        \to \I_X \to 0
    \end{equation*}
    Since this resolution satisfies linear twist growth, Theorem \ref{lineartoreg} applies to find that $(2,1) + \N^2\subseteq \reg(\I_X)$. Using \textit{Macaulay2}, we can compute the actual regularity region to be 
    $\reg(\I_X) = ((1,0) + \N^2) \cup ((0,1) + \N^2)$. 
\end{example}

\begin{remark}
    The resolution in Example \ref{twopointsinP1xP1} comes from sheafifying a virtual resolution of the ideal $I_X$. Specifically, the resolution corresponds to the virtual resolution of a pair $(I_X,(1,0))$ (see \cite{berkesch2017virtual}*{Theorem 3.1}). We could also get the virtual resolution of the pair $(I_X, (0,1))$ to get the resolution 
    \begin{equation*}
        0 \to
        \begin{matrix}
            \O_{\P^1 \times \P^1}(-1,-2)^2
        \end{matrix} 
        \to 
        \begin{matrix}
            \O_{\P^1 \times \P^1}(-1,-1)^2\\ \oplus \\ 
            \O_{\P^1 \times \P^1}(0,-2)
        \end{matrix} 
        \to \I_X \to 0
    \end{equation*}
    and another bound $(1,2)+\N^2 \subseteq \reg(\I_X)$.
\end{remark}

\begin{figure}
    \newcommand{\makegrid}{
    \foreach \x in {-2,...,6}
    \foreach \y in {-2,...,6}
        { \fill[gray,fill=gray] (\x,\y) circle (1.5pt); }
    \draw[-,  semithick] (-2,0)--(6,0);
    \draw[-,  semithick] (0,-2)--(0,6);
    }
    \begin{tikzpicture}[scale=0.5]
        \makegrid
        
        \path[fill=OrangeRed!45, opacity=0.5] (2,1)--(2,6)--(6,6)--(6,1)--(2,1);
        \draw[->, ultra thick,OrangeRed] (2,1)--(2,6);
        \draw[->, ultra thick,OrangeRed] (2,1)--(6,1);
        \fill[OrangeRed,fill=Black] (2,1) circle (6pt);

        \path[fill=PineGreen!45, opacity=0.5] (0,1)--(0,6)--(6,6)--(6,0)--(1,0)--(1,1)--(0,1);
        \draw[->, ultra thick,PineGreen] (0,1)--(0,6);
        \draw[-, ultra thick,PineGreen] (0,1)--(1,1);
        \draw[-, ultra thick,PineGreen] (1,1)--(1,0);
        \draw[->, ultra thick,PineGreen] (1,0)--(6,0);
        \fill[PineGreen,fill=Black] (0,1) circle (6pt);
        \fill[PineGreen,fill=Black] (1,0) circle (6pt);
    \end{tikzpicture}
    \caption{The two regions inside $\Pic(\P^1\times \P^1) \cong \Z^2$ from Example \ref{twopointsinP1xP1}; $(2,1) + \N^2$ in red and $\reg(\I_X)$ in green.}
\end{figure}
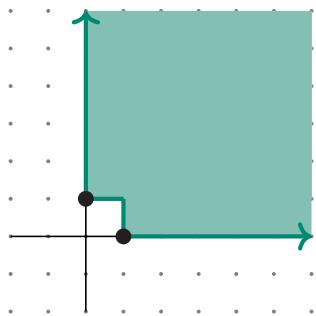

\begin{remark}\label{Comparison of linearity}
    Theorem \ref{lineartoreg} compares closely to both Theorem 2.9 in \cite{berkesch2017virtual} and Theorem A in \cite{SayrafiBruceMultigradedRegularity}. Linear twist growth closely compares with \textit{linearity} in \cite{SayrafiBruceMultigradedRegularity} --  Complexes that satisfy linearity from \cite{SayrafiBruceMultigradedRegularity} will have linear twist growth after sheafifying. However, the same may not be true for the notion of quasilinearity from \cite{SayrafiBruceMultigradedRegularity}, which is the same as having degrees inside of the regions $\Delta_i + \N^r$ in \cite{berkesch2017virtual}. So while we shouldn't expect a strict converse to Theorem \ref{lineartoreg} to hold, relaxing our definition of linear twist growth would allow one. Proposition 1.7 in \cite{Eisenbud2014TateRF} is closest to being a converse in current literature.
\end{remark}

\begin{remark}
    The proof of Theorem \ref{lineartoreg} uses the decomposition of $\E$ as a sum of line bundles in an essential way since we can then make use of the fact that $\O_{\P^\r}$ is $\bm{0}$-regular. It would be interesting to know if Theorem \ref{lineartoreg} still holds if we allow $\E_i$ to be coherent sheaves.
\end{remark}

\section{Main Result and Proof}\label{Main Result and Proof}

Throughout this section let $C\subseteq \P^\r$ be a reduced curve with smooth normalization $\widetilde{C}$ and natural nondegenerate map $p\colon \widetilde{C} \to \P^\r$. Let $\A$ be a line bundle on $\widetilde{C}$ and define
\begin{equation*}
    h_{\A}^i(\bm{m}) \coloneqq \dim H^i(\widetilde{C},\wedge^{\m} \M \otimes \A) = \dim H^i(\widetilde{C},p^*\Omega^\m_{\P^\r}(\m) \otimes \A). 
\end{equation*}
We now prove the key towards establishing Theorem \ref{maintheorem} for possibly singular curves. Proposition \ref{maintheoremsingular} says that a complex that resolves $\Fitt(p_*\A)$ off a set of dimension $\leq 1$ is controlled by the positivity of $\wedge^\m \M$. If $C$ happens to be smooth then $\Fitt(p_*\A) \cong \I_C$ and this complex becomes a resolution. 

\begin{proposition}\label{maintheoremsingular}
    Suppose that $h_{\A}^1(\m) = 0$ for all $|\m| = 2$. Then we have an Eagon-Northcott complex 
    \begin{equation*}
    \cdots \to \E_2 \to  \E_1 \to \E_0 \to \Fitt(p_*\A) \to 0
    \end{equation*}
    where $\E_i$ are direct sums of line bundles $\O_{\P^\r}(-\m_i)$ with $\m_i \in \N^n$ and $|\m_i| = h^0(\widetilde{C},\A)+i$.
\end{proposition}
\begin{remark}
    Proposition \ref{maintheoremsingular} specializes to Proposition 1.2 in \cite{gruson1983theorem} in the case $n=1$.
\end{remark}
\begin{proof}
    The first step is to show that there is an exact sequence 
    \begin{equation}\label{linear presentation}
        \bigoplus_{|\m| = 1} \O_{\P^\r}(-\m)^{\oplus h_{\A}^0(\m)}  \stackrel{u}{\to}  H^0(\widetilde{C},\A) \otimes \O_{\P^\r} \to  p_*\A  \to 0
    \end{equation}
    of sheaves on $\P^\r$ -- in effect, we are showing the multilinear analogue of $p_*\A$ having a linear presentation. Once we do this, the Eagon-Northcott complex constructed from $u$ is the desired complex. We will obtain (\ref{linear presentation}) by computing the pushforward of a certain Koszul resolution of $\Gamma_p \subseteq \widetilde{C} \times \P^\r$, the graph of $\widetilde{C}$ along $p$.

    In order to construct this Koszul resolution, consider the following morphisms:
    \begin{equation*}
    \begin{tikzcd}[row sep = tiny]
        \widetilde{C} \ar[r, "\text{id} \times p"] & \widetilde{C} \times \P^\r \ar[r, "p \times \text{id}"] & \P^\r \times \P^\r
    \end{tikzcd}    
    \end{equation*}
    Sitting inside of $\P^\r \times \P^\r$ there is the diagonal $\Delta$, the product of the diagonals $\Delta_i \subseteq \P^{r_i} \times \P^{r_i}$ in each factor. The graph $\Gamma_p$ is defined scheme-theoretically by the pullback of $\Delta$ across the map $p\times \text{id}$. Famously, due to Be\u{i}linson \cite{Beilinson1978CoherentSO}, the structure sheaf $\O_{\Delta_i}$ of each diagonal has the locally free resolution $B_i$:
    \begin{equation*}
        \left[ 0 \to \Omega^i_{\P^{r_i}}(i) \boxtimes \O_{\P^{r_i}}(-i) \to \Omega^{i-1}_{\P^{r_i}}(i-1) \boxtimes \O_{\P^{r_i}}(-i+1) \to \cdots \to \Omega_{\P^{r_i}}(1) \boxtimes \O_{\P^{r_i}}(-1) \to \O_{\P^{r_i}\times \P^{r_i}} \right].
    \end{equation*}
    
    Therefore by Lemma 2.7.3 in \cite{Weibel}, $\O_\Delta$ is resolved by the tensor complex $B = \boxtimes_{i=0}^n B_i$:
    \begin{equation*}
        \left[\bigoplus_{|\m| = \max\{r_i\}} \Omega^\m_{\P^\r}(\m) \boxtimes \O_{\P^\r}(-\m) \to \cdots \to \bigoplus_{|\m| = 1} \Omega^\m_{\P^\r}(\m) \boxtimes \O_{\P^\r}(-\m) \to \O_{\P^\r \times \P^\r} \right]
    \end{equation*}
    To see an explicit argument of this fact, see Lemma 2.1 in \cite{ berkesch2017virtual}. Since pullback is an exact functor on vector bundles, $\mc{K} \coloneqq (p \times \text{id})^*(B)$ is a resolution of $(p \times \text{id})^*(\O_\Delta) = \O_{\Gamma_p}$ in  $\widetilde{C} \times \P^\r$. 

    Consider the two natural projections on $\widetilde{C} \times \P^\r$:
    \begin{equation*}
        \begin{tikzcd}[row sep = tiny]
            & \widetilde{C} \times \P^\r \ar[ld, "\pi"'] \ar[rd, "f"]\\
            \widetilde{C} & & \P^\r
        \end{tikzcd}
    \end{equation*}

    Since $\A$ is a line bundle on $\widetilde{C}$, tensoring the complex $\mc{K}$ by the pullback $\pi^*\A$ is still exact. $\mc{K} \otimes \pi^* \A$ resolves $\O_{\Gamma_p} \otimes \pi^* \A$, therefore the total derived pushforward $\R f_*(\mc{K} \otimes \pi^*\A)$ is quasi-isomorphic to the object $\R f_*(\O_{\Gamma_p} \otimes \pi^* \A)$ considered as a complex concentrated in degree zero. The following argument confirms that $\R f_*(\O_{\Gamma_p} \otimes \pi^* \A)$ is the first object of (\ref{linear presentation}). 
    
    \begin{adjustwidth}{3em}{3em}
        \textbf{Claim.} $\R f_*(\O_{\Gamma_p} \otimes \pi^* \A) \cong p_*\A$
        \begin{proof}
            Since $\O_{\Gamma_p} = \pi^* \O_{\widetilde{C}}$, we have
        \begin{align*}
            \R f_*(\O_{\Gamma_p} \otimes \pi^*\A) &= \R f_*(\pi^* \O_{\widetilde{C}} \otimes \pi^*\A)  \\
            &=   \R f_*(\pi^* (\O_{\widetilde{C}} \otimes \A)) \\
            &=\R f_*(\pi^* \A)
        \end{align*}
        Then by flat base change \cite[\href{https://stacks.math.columbia.edu/tag/02KH}{Tag 02KH}]{stacks-project} across the following diagram
        \begin{equation*}
            \begin{tikzcd}
                \Gamma_p \ar[r, "\pi"] \ar[d,"f"'] & \widetilde{C} \ar[d, "p"]\\
                \P^\r \ar[r, "id"] & \P^\r
            \end{tikzcd}
        \end{equation*}
        we have that $\R f_*(\O_{\Gamma_p} \otimes \pi^* \A) = p_* \A$. Here, $\pi$ and $f$ are the same projection maps defined above but restricted to the graph $\Gamma_p \subseteq \widetilde{C} \times \P^\r$.
        \end{proof}
    \end{adjustwidth}
    To verify (\ref{linear presentation}), it only remains to compute the first two terms of $\R f_*(\mc{K} \otimes \pi^*\A)$, which we will argue remains exact. The derived pushforward of $\R f_*(\mc{K} \otimes \pi^*\A)$ can be computed by a (second-quadrant) Be\u{i}linson-like spectral sequence
    \[ {}_{\uparrow}E^{\ell,q}_1\colon  R^qf_*((\mc{K} \otimes \pi^*\A)_\ell) \implies p_*\A \]
    where $(\mc{K} \otimes \pi^*\A)_\ell$ is the $\ell^\text{th}$ term of the complex $\mc{K} \otimes \pi^*\A$ (\cite{OkonekVectorBundles}*{II.3.1} or \cite{lazarsfeld2017positivity}*{Appendix(B.2)} or \cite{SzpiroSS}*{\S2.1}). 
    By the projection and K\"{u}nneth formulas,
    \begin{align*}
        R^qf_*((\mc{K} \otimes \pi^*\A)_\ell) &= R^qf_*\left(\bigoplus_{|\m| = \ell} (p^*\Omega^\m_{\P^\r}(\m) \otimes \A) \boxtimes \O_{\P^\r}(-\m)\right)\\
         &= \bigoplus_{|\m| = \ell} R^qf_*(p^*\Omega^\m_{\P^\r}(\m) \otimes \A) \otimes \O_{\P^\r}(-\m)\\
         &= \bigoplus_{|\m| = \ell} H^q(\widetilde{C}, \wedge^\m \M \otimes \A) \otimes \O_{\P^\r}(-\m)\\
        &= \bigoplus_{|\m| = \ell} \O_{\P^\r}(-\m)^{\oplus h_{\A}^q(\m)}.
    \end{align*}

    Thus ${}_{\uparrow}E_1$ looks like
    \begin{equation*}\label{eq:derived-E1}
        \begin{tikzcd}[column sep=small, row sep=tiny]
        {} & 0 &  0 &  0 & \cdots \\
        {} & H^1(\widetilde{C},\A) \otimes \O_{\P^\r} &  \bigoplus_{|\m|=1} \O_{\P^r}(-\m)^{\oplus h_{\A}^1(\m)} &  \bigoplus_{|\m|=2} \O_{\P^r}(-\m)^{\oplus h_{\A}^1(\m)} & \cdots \\
        {} & H^0(\widetilde{C},\A) \otimes \O_{\P^\r} &  \bigoplus_{|\m|=1} \O_{\P^r}(-\m)^{\oplus h_{\A}^0(\m)} &  \bigoplus_{|\m|=2} \O_{\P^r}(-\m)^{\oplus h_{\A}^0(\m)} & \cdots \\
        {} & {} & {}
        \arrow[from=3-4, to=3-3]
        \arrow[from=3-3, to=3-2]
        \arrow[from=3-5, to=3-4]
        \arrow[from=2-5, to=2-4]
        \arrow[from=2-4, to=2-3]
        \arrow[from=2-3, to=2-2]
        \arrow[from=1-5, to=1-4]
        \arrow[from=1-4, to=1-3] 
        \arrow[from=1-3, to=1-2]
        \arrow[crossing over, shift left=18, shorten <=-28pt, shorten >=-18pt, dashed, no head, from=3-2, to=1-2]
        \arrow[shift right=5, shorten <=-99pt, shorten >=-20pt, dashed, no head, from=3-2, to=3-5]
        \end{tikzcd}
      \end{equation*}
      The $H^2$ row vanishes since $\widetilde{C}$ is 1-dimensional. In the limit, the only other differential going into $H^0(\widetilde{C},\A) \otimes \O_{\P^\r}$ is the one on the second page ${}_{\uparrow}E_2$:
    \begin{equation*}\label{eq:derived-E1}
        \begin{tikzcd}[column sep=small, row sep=tiny]
        {} & 0 &  0 &  0 & \cdots \\
        {} & H^1(\widetilde{C},\A) \otimes \O_{\P^\r} &  \bigoplus_{|\m|=1} \O_{\P^r}(-\m)^{\oplus h_{\A}^1(\m)} &  \bigoplus_{|\m|=2} \O_{\P^r}(-\m)^{\oplus h_{\A}^1(\m)} & \cdots \\
        {} & H^0(\widetilde{C},\A) \otimes \O_{\P^\r} &  \bigoplus_{|\m|=1} \O_{\P^r}(-\m)^{\oplus h_{\A}^0(\m)} &  \bigoplus_{|\m|=2} \O_{\P^r}(-\m)^{\oplus h_{\A}^0(\m)} & \cdots \\
        {} & {} & {}
        \arrow[from=2-4, to=3-2]
        \arrow[crossing over, shift left=18, shorten <=-28pt, shorten >=-18pt, dashed, no head, from=3-2, to=1-2]
        \arrow[shift right=5, shorten <=-99pt, shorten >=-20pt, dashed, no head, from=3-2, to=3-5]
        \end{tikzcd}
      \end{equation*}
    However, since $h_{\A}^1(\m)$ vanishes for all $|\m|=2$ by assumption, we finally establish the exact sequence (\ref{linear presentation}). The Eagon-Northcott complex \ref{EagonNorthcott} constructed from $u$ takes the form
    \begin{equation}\label{FittingEN}
        \cdots \to \bigoplus_{|\m| = h^0(\widetilde{C},\A)+1} \O_{\P^r}(-\m)^{\oplus M(\m)} \to \bigoplus_{|\m| = h^0(\widetilde{C},\A)} \O_{\P^\r}(-\m)^{\oplus M(\m)} \stackrel{\wedge^{h^0(\widetilde{C},\A)}u}{\longrightarrow} \Fitt(p_*\A) \to 0 
    \end{equation}   
    where $\wedge^{h^0(\widetilde{C},\A)} u$ is surjective. The ranks $M(\m)$ are explicitly computed in the proof of Proposition \ref{regularityCor}.
\end{proof}


As long as $\A$ is chosen to be positive enough, the vanishings $h^1_\A(\m)=0$ where $|\m|=2$ are automatic by Serre vanishing -- we want to be able to choose $\A$ well. Lemma \ref{h^0(A)} establishes an upper bound for the positivity of $\A$.

\begin{lemma}\label{h^0(A)}
    There exists a line bundle $\A$ on $\widetilde{C}$ with \[h^0(\widetilde{C},\A) = \max\Set*{d_i + d_j - r_i - r_j \given i \neq j} + 2\] such that $h^1_\A(\m) = 0$ for all $|\m| = 2$.
\end{lemma}
\begin{proof}
    As stated, as long as $h^0(\widetilde{C},\A)$ is high enough, we get the result via Serre's vanishing. We will check that $h^0(\widetilde{C},\A) = \max\Set*{d_i + d_j - r_i - r_j \given i \neq j} + 2$ suffices for all choices $\m = \bm{e}_i + \bm{e}_j$ by splitting the analysis into two cases: either $i=j$ or $i\neq j$. The first case is handled by applying the original lemma to each factor, giving $h^0(\widetilde{C},\A)= \max\Set*{d_i - r_i}+2$ \cite[Lemma 1.7]{gruson1983theorem}. We will now show that the second case requires $h^0(\widetilde{C},\A) = \max\Set*{d_i + d_j - r_i - r_j \given i \neq j} + 2$. In situations where we have more than one factor, this case always dominates; since the mapping of $\widetilde{C}$ is nondegenerate in each factor, we have $d_i \geq r_i$ (Proposition 5.3, \cite{GeometryOfSyzygies}).
    
    Without loss of generality, suppose $i=1$ and $j=2$ maximize $\Set*{d_i + d_j - r_i - r_j \given i \neq j}$. Let $\M_1$ and $\M_2$ denote the first and second factor, respectively, of the tensor product $\wedge^\m \M = \wedge^{\bm{e}_1+\bm{e}_2} \M = p^*\O_{\P^{r_1}}(1) \otimes p^*\O_{\P^{r_2}}(1)$. As in the original proof, $\M_1$ and $\M_2$ have filtrations by vector bundles
    \begin{equation*}
        \M_i = F_i^1 \supseteq F^2_i \supseteq \cdots \supseteq F_i^{r_i} \supseteq F_i^{r_i+1} = 0
    \end{equation*}
    whose successive quotients $Q_i^j \coloneqq F_i^j/F_i^{j+1}$ are line bundles of strictly negative degree (see Lemma 1.7 in \cite{gruson1983theorem}, or Proposition 5.13 in \cite{GeometryOfSyzygies}). To complete our claim 
    \begin{equation*}
        h^1_\A(\bm{e}_1 + \bm{e}_2)=\dim H^1(\widetilde{C}, \M_1 \otimes \M_2 \otimes \A) = 0,
    \end{equation*}
    we see by splitting up the filtration $F^i_1$ of $\M_1$ into short exact sequences that it suffices to show the vanishing $H^1(\widetilde{C},Q^i_1 \otimes \M_2 \otimes \A ) = 0$ for all $1\leq i \leq r_i$. Similarly, by using the filtration $F^j_2$ of $\M_2$ we see it further suffices to show 
    \[ H^1(\widetilde{C}, Q_1^i \otimes Q_2^j \otimes \A) = 0 \hspace{0.2in} \text{ for all } 1\leq i \leq r_1 \text{ and } 1 \leq j \leq r_2 \tag{$*$}.\]

    \noindent Since 
    \[ \deg(\M_1) = \sum^{r_1}_{k=1} \deg Q^k_1 = -d_1\] 
    and each line bundle $Q_1^k$ has degree $< 0$, we find that for any fixed $i$ we have the bound
    \begin{equation*}
        \deg Q_1^i = -d_1 - \sum_{k\neq i}^{r_1} \deg Q_1^k \geq -d_1 - (1-r_1) = r_1 - d_1 - 1.
    \end{equation*}
    Of course, the same argument works to get a bound on $\deg Q_2^j$, and putting these results together yields  
    \[ \deg Q^i_1 \otimes Q^j_2 \geq r_1+r_2- d_1 - d_2 -2  \hspace{0.2in} \text{ for all } 1\leq i \leq r_1 \text{ and } 1 \leq j \leq r_2 \]
    Since a generic line bundle of degree $\geq g-1$ is non-special (as a consequence of geometric Riemann-Roch in \S I.2 in \cite{GeometryofAlgebraicCurves} or of Theorem 8.5(1) in \cite{GeometryOfSyzygies}), $(*)$ will be satisfied if $\A$ is a sufficiently general line bundle of degree $g + d_1 + d_2  - r_1 - r_2 + 1$.
    
    Additionally, with this choice of $\A$ one can verify that $\deg \A \geq g - 1$. That is, $\A$ is also non-special and $h^1(\widetilde{C},\A) = 0$. Finally, we can calculate $h^0(\widetilde{C},\A)$ using Riemann-Roch:
    \begin{align*}
        h^0(\widetilde{C},\A) &= h^1(\widetilde{C},\A) + \deg \A + 1 - g \\
        &= d_1 + d_2 - r_1 - r_2 + 2
        \qedhere
    \end{align*}
\end{proof}

\begin{remark}\label{Why statement doesnt reduce to n=1}
    Lemma \ref{h^0(A)} is the primary reason that Proposition \ref{regularityCor} does not reduce to Theorem 1.1 in \cite{gruson1983theorem}. In Theorem 1.1, only a single vanishing $H^1(\widetilde{C}, \wedge^2 \M \otimes \A)=0$ was needed. In contrast, for Lemma \ref{h^0(A)} we must choose $\A$ positive enough to ensure $H^1(\widetilde{C}, \wedge^\m \M \otimes \A)=0$ for all choices of $\m = \bm{e}_i + \bm{e}_j$ satisfying $|\m|=2$. In particular, this necessitates new casework handling $i\neq j$. 
\end{remark}

Proposition \ref{maintheoremsingular} and Lemma \ref{h^0(A)} together give a statement stronger than Theorem \ref{maintheorem}. We finish this section by making the reduction to the smooth case explicit.

\begin{proof}[Proof of Theorem \ref{maintheorem}]
    The sheaf $p_*\A$ is locally isomorphic to $\O_C$ except at the points where $p$ fails to be an isomorphism. In particular, this implies that the support of $p_*\A$ is contained in $C$. But this occurs only at the finitely many singular points of $C$, so $p$ is a closed embedding and around any point we can pick a $U$ such that 
    \begin{align*}
        \res[U]{\Fitt(p_*\A)} &= \Fitt(\res[U]{p_*\A})\\
        &= \Fitt(\O_U)\\
        &= \I_U
    \end{align*}
    Therefore $\Fitt(p_*\A)=\I_C$. Since $C$ is smooth we know $C\cong \widetilde{C}$ and by Fact \ref{ENfact} the complex is actually exact.
\end{proof}

\section{Multiregularity of Curves in Products of Projective Spaces}\label{Regularity of Curves}
An important classical problem is to bounds on the degrees of the defining questions of subvarieties of projective spaces, or more generally to find bounds on their Castelnuovo-Mumford regularity \cites{gruson1983theorem,Ein1993SyzygiesAK, Kwak1998CastelnuovoRF,Kwak2000GenericPT,McCullough2017CounterexamplesTT}. Of particular importance is the aforementioned theorem of Gruson, Lazarsfeld, and Peskine \cite{gruson1983theorem} showing that when $C\subseteq \P^r$ is an integral curve of degree $d$ then $\reg(C) = \reg(\I_C) \leq d - r + 2$. Lozovanu later handled the case for smooth curves in products of two projective spaces (provided neither are $\P^1$) with birational projections \cite{Lozovanu2008RegularityOS}. In keeping with the growing interest in finding bounds on multigraded regularity regions \cites{Bruce2022BoundsOM, Maclagan2003UniformBO}, in this section we use Proposition \ref{maintheoremsingular} to prove Proposition \ref{regularityCor}, the closest analogue for curves in arbitrary products of projective spaces. To simplify the flow of the proof of Proposition \ref{regularityCor}, we set aside a necessary computation in Lemma \ref{h^0(e)}.

\begin{lemma} \label{h^0(e)}
    Under the choice of $\A$ in Lemma \ref{h^0(A)} we have
    \begin{equation*}
        h^0_\A(\bm{e}_k) = -d_k + r_kh^0(\widetilde{C},\A).
    \end{equation*}
\end{lemma}
\begin{proof}
    By the Hirzebruch-Riemenn-Roch Theorem, we know that 
    \begin{align}
        h^0_\A(\bm{e}_k) &= h^1_\A(\bm{e}_k) + c_1(\M_k \otimes \A) + r_k(1-g)\nonumber \\
        &=  h^1_\A(\bm{e}_k) + c_1(\M_k) + r_kc_1(\A) + r_k(1-g) \label{HRRthm}
    \end{align} 
    Now, using the calculation in Lemma \ref{h^0(A)} and the fact that $\A$ is a line bundle,
    \begin{equation*}
        c_1(\A) = \deg \A = g + \max\Set*{d_i + d_j - r_i - r_j \given i \neq j} + 1
    \end{equation*}
    Lastly, we compute $c_1(\M_k)$. Since taking determinants is multiplicative across exact sequences,
    \begin{equation*}
        \det(\M_k) \otimes \L_k = \det(\O_{\widetilde{C}}^{\oplus r_k}) \hspace{0.1in} \implies \hspace{0.1in} \det(\M_k) = \L_k^{-1}.
    \end{equation*}
    Therefore, using that fact that $c_1(\M_k) = c_1(\det \M_k)$ \cite{FultonIntersectionTheory}*{Remark 3.2.3(c)} we find 
    \begin{equation*}
        c_1(\M_k) = c_1(\det \M_k) = c_1(\L_k^{-1}) = - \deg \L_k = -d_k.
    \end{equation*}
    Lastly, since $\A$ has been chosen to be much more positive than $\deg(\M_k) = -d_k$, we have $h^1(\bm{e}_k) = h^1(\M_k \otimes \A) = 0$. More precisely, one can find a filtration $F^\bullet$ of $\M_k$ exactly as in Lemma \ref{h^0(A)} to find that this vanishing occurs when $\deg \A$ larger than $g+d_k-r_k$, which is true of our choice of $\A$. Plugging all of these values back into (\ref{HRRthm}) gives the result.
\end{proof}

With the calculations from Lemmas \ref{h^0(A)} and \ref{h^0(e)} in hand, we can prove Proposition \ref{regularityCor}.

\begin{proof}[Proof of Proposition \ref{regularityCor}]
    Define $a \coloneqq h^0(\widetilde{C},\A)$ and let $\Fitt(p_*\A)\subseteq \O_{\P^\r}$ denote the zeroth Fitting ideal sheaf of $p_*\A$ computed from (\ref{linear presentation}), i.e. the image of $\wedge^{a}{u}$. Since $p_*\A$ is locally isomorphic to $\O_C$ except at the finitely many singular points of $C$, we know that $p_*\A$ is supported on $C$. This means that the subscheme defined by the sheaf of ideals $\Fitt(p_*\A)$ coincides set-theoretically with $C$ since 
    \[ \Supp(\O_{\Fitt(p_*\A)}) = \Supp(\coker u) = \Supp(p_*\A) \supseteq C\]
    Since $C$ is reduced, this implies that 
    \[ \Fitt(p_*\A) \subseteq \I_C,\]
    giving rise to the exact sequence
    \[ 0 \to \Fitt(p_*\A) \to \I_C \to \I_C/\Fitt(p_*\A) \to 0.\]
    Given that $\I_C/\Fitt(p_*\A)$ is supported on a finite set (the finitely many singular points of $C$), running the long exact sequence shows us that $\reg \I_C \subseteq \reg \Fitt(p_*\A)$. So the proof is finished after bounding $\reg \Fitt(p_*\A)$, which we can do using complex (\ref{FittingEN}).

    The complex (\ref{FittingEN}) is exact off of $C$, and has linear twist growth. By Theorem \ref{lineartoreg}, it remains to verify that the regularity region of the first term in complex (\ref{FittingEN}) is $(a_1,\dots, a_n) + \N^n$ where $a_k = \min\{-d_k+r_ka, a\}$. The regularity region of this term can be computed very explicitly:
    \begin{align*}
        \reg\left(\bigoplus_{|\m| = a} \O_{\P^\r}(-\m)^{\oplus M(\m)} \right) &= \bigcap_{|\m|=a} \reg \O_{\P^\r}(-\m)^{\oplus M(\m)}\\
        &= \bigcap_{\substack{|\m|=a\\ \M(\m) \neq 0}} \m + \N^n.
    \end{align*}

    This intersection will be a region of the form $\bm{n} +\N^n$ where $\bm{n}$ is the componentwise maximum of vectors in $S=\Set*{\m \given M(\m)\neq 0 \text{ and } |\m| = a}$, the corners of the regions appearing in the intersection. Let $s_k$ denote the $k^{\text{th}}$ component maximum of $S$. In the event that $a\bm{e}_k$ is in $S$, we have $s_k = a$. It only remains to show that in the case  $a\bm{e}_k$ is not in $S$, this maximum is $s_k=h^0_\A(\bm{e}_k) = -d_k +r_ka$.

    The ranks $M(\m)$ in this complex are determined by wedge products of the source and target of $u$:
    \[ M(\m) =  \begin{pmatrix} h^0_{\A}(\bm{e}_1) \\ m_1 \end{pmatrix}\cdots \begin{pmatrix} h^0_{\A}(\bm{e}_n) \\ m_n \end{pmatrix}. \]
    Therefore $M(\m)$ vanishes if and only if there is an index $1\leq j\leq n$ such that $m_j > h^0_\A(\bm{e}_j)$. Given that $ae_k$ is not in $S$, we know that $M(a\bm{e}_k)=0$ and thus $s_k \leq h^0_\A(\bm{e}_k) < a$. To verify that in fact $s_k = h^0_\A(\bm{e}_k)$, it suffices to show that we can find $\m$ with $m_k= h^0_\A(\bm{e}_k)$ satisfying $M(\m) \neq 0$ and $|\m|=a$.

    Without loss of generality, letting $k=1$, this means we need to find $\m$ such that 
    \begin{equation*}
        M(\m) = \begin{pmatrix} h_\A^0(\bm{e}_2) \\ m_2 \end{pmatrix}\cdots \begin{pmatrix} h_\A^0(\bm{e}_n) \\ m_n \end{pmatrix}.
    \end{equation*}
    It is possible to choose $\m$ so that $M(\m) \neq 0$ as long as the sum of the lower entries in the binomial coefficients is less than the sum of the top entries. Using $|\m|=a$, this is equivalent to checking $a\leq \sum_{j=1}^n h^0_\A(\bm{e}_j)$.

    Taking the alternating sum of ranks from the linear presentation (\ref{linear presentation}) gives
    \begin{equation*}
        a = \sum_{j=1}^n h_\A^0(\bm{e}_j) + 1 - \rank(\ker u).
    \end{equation*}
    If $\rank(\ker u)> 0$ then we are done. Otherwise, if $\rank(\ker u)=0$ then the linear presentation (\ref{linear presentation}) was actually a resolution of $p_*\A$. This means that the support of $p_*\A$ has codimension $\leq 1$. Since $C$ has dimension $1$, we know $\dim \P^{\r} = 2$ and so the only possibility when $n>1$ is $\P^\r = \P^1\times \P^1$, which has been excluded by assumption.
\end{proof}

A counterexample to the case of $\P^1\times \P^1$ is given in Example \ref{LozovanuExample}. Although Proposition \ref{regularityCor} considers a fairly general class of curves in $\P^\r$, the regularity bound $(a_1,\dots, a_n)+\N^n$ is far from optimal -- for example, the regularity bound is dominated by $a=h^0(\widetilde{C},\A)$, which is invariant with respect to exchanging the coordinates of the product $\P^\r$ even though regularity is not. Given other resolutions and precise results linking resolutions to regularity, the bound in Proposition \ref{regularityCor} may potentially be improved. 

\begin{remark}
    As in \cite{gruson1983theorem}, only slight modifications of the previous results are necessary to give a bound on $C$ that is reduced but possibly reducible. This proceeds by breaking $C$ up into irreducible components and applying the lemma component by component.
\end{remark}

\section{Examples}\label{Examples}

\begin{example}\label{standardexample}
    Let $\I_C$ be the ideal sheaf corresponding to \[ \left\langle x_0^2y_0^2+x_1^2y_1^2+x_0x_1y_2^2, x_0^3y_2+x_1^3(y_0+y_1) \right\rangle \] defining a degree $(2,8)$ smooth hyperelliptic curve $C$ of genus 4 embedded into $\P^1\times\P^2$. Applying Proposition \ref{regularityCor} tells us that $(7,9)+\N^2\subseteq \reg \I_C$, which can be confirmed using \cite{SayrafiBruceMultigradedRegularity}*{Theorem A}. We can not apply \cite{Lozovanu2008RegularityOS}*{Theorem A} since one of the factors is $\P^1$. The actual regularity region $\reg \I_C$, as computed in \textit{Macaulay2}, is depicted in Figure \ref{standardexamplefigure}.
\end{example}

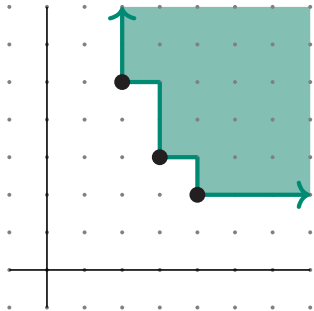
\begin{figure}
    \newcommand{\makegrid}{
    \foreach \x in {-1,...,7}
    \foreach \y in {-1,...,7}
        { \fill[gray,fill=gray] (\x,\y) circle (1.5pt); }
    \draw[-,  semithick] (-1,0)--(7,0);
    \draw[-,  semithick] (0,-1)--(0,7);
    }
    \begin{tikzpicture}[scale=0.5]
        \path[fill=PineGreen!45] (2,5)--(2,7)--(7,7)--(7,2)--(4,2)--(4,3)--(3,3)--(3,5)--(2,5);
        \makegrid
    
        \draw[->, ultra thick,PineGreen] (2,5)--(2,7);
        \draw[-, ultra thick,PineGreen] (2,5)--(3,5);
        \draw[-, ultra thick,PineGreen] (3,5)--(3,3);
        \draw[-, ultra thick,PineGreen] (3,3)--(4,3);
        \draw[-, ultra thick,PineGreen] (4,3)--(4,2);
        \draw[->, ultra thick,PineGreen] (4,2)--(7,2);
        \fill[TealBlue,fill=Black] (2,5) circle (6pt);
        \fill[TealBlue,fill=Black] (3,3) circle (6pt);
        \fill[TealBlue,fill=Black] (4,2) circle (6pt);
        \end{tikzpicture}
    \caption{Depicts $\reg \I_C$ of the curve defined in Example \ref{standardexample}.} \label{standardexamplefigure}
\end{figure}

\begin{example}
    Let $\P^1$ embed into $\P^{d_1} \times \cdots \times \P^{d_n}$ by the $d_i$-uple embedding in each factor. Then Theorem \ref{maintheorem} tells us that $C$ is $(2,\dots, 2)$-regular, while the actual regularity is $\bm{1} + \N^n$.
\end{example}

\begin{example}(Counterexample for $\P^1 \times \P^1$)\label{LozovanuExample}
    Let $\I_C$ be the ideal sheaf corresponding to \[ \left\langle x_0^3y_0^2+x_1^3y_1^2+x_0x_1^2y_0y_1 \right\rangle, \] defining a smooth curve $C$ of genus 2 and degree $(3,2)$ embedded into $\P^1\times\P^1$. Since this is a hypersurface, the ideal sheaf is $\I_C = \O_{\P^{1} \times \P^1}(-3,-2)$ and has regularity $(3,2) + \N^2$. However, both Proposition \ref{regularityCor} and \cite{Lozovanu2008RegularityOS} would say that $C$ is $(2,3)$-regular, which is false. 
\end{example}

\bibliography{bib}{}

\end{document}